\documentclass[12pt]{amsart}
\usepackage{amsmath,amsthm,amsfonts,amssymb,mathrsfs}
\date{\today}

\usepackage[cp1251]{inputenc}         

\usepackage{amsmath,amsthm,amsfonts,amssymb,mathrsfs}
\usepackage{eucal}

\usepackage{color}

\usepackage{amssymb,cite}
\usepackage[colorlinks,plainpages,citecolor=magenta, linkcolor=blue, backref]{hyperref}

\usepackage{hyperref}

\newtheorem{theorem}{Theorem}

\newtheorem{proposition}{Proposition}
\newtheorem{corollary}{Corollary}

\theoremstyle{definition}
\newtheorem{definition}{Definition}
\newtheorem{example}{Example}

\theoremstyle{remark}
\newtheorem{remark}{Remark}
\newtheorem{question}{Question}

\numberwithin{equation}{section}

\begin{document}

\title[On topologization of subsemigroups of the bicyclic monoid]{On topologization of subsemigroups of the bicyclic monoid}

\author[A. Chornenka and O.~Gutik]{Adriana Chornenka and Oleg~Gutik}
\address{Ivan Franko National University of Lviv, Universytetska 1, Lviv, 79000, Ukraine}
\email{adriana.chornenka@lnu.edu.ua, oleg.gutik@lnu.edu.ua, ogutik@gmail.com}


\begin{abstract}
We show that if a subsemigroup $S$ of the bicyclic monoid ${\mathscr{C}}(p,q)$  contains infinitely many idempotents then $S$ admits only the  discrete Hausdorff shift-continuous topology. Also we proof that every right-continuous (left-continuous\emph) Hausdorff Baire topology  on the semigroup $\mathscr{C}_+(a,b)$ $(\mathscr{C}_-(a,b))$ is discrete and the same statement holds for the bicyclic monoid.
\end{abstract}

\maketitle

In this paper we shall follow the terminology of \cite{Carruth-Hildebrant-Koch=1983, Carruth-Hildebrant-Koch=1986, Clifford-Preston=1961, Clifford-Preston=1967, Engelking=1989, Lawson=1998, Ruppert=1984}. By $\omega$ and $\mathbb{N}$ we denote the set of non-negative integers and the set of positive integers, respectively.

\smallskip



A semigroup $S$ is called {\it inverse} if for any
element $x\in S$ there exists a unique $x^{-1}\in S$ such that
$xx^{-1}x=x$ and $x^{-1}xx^{-1}=x^{-1}$. The element $x^{-1}$ is
called the {\it inverse of} $x\in S$. If $S$ is an inverse
semigroup, then the function $\operatorname{inv}\colon S\to S$
which assigns to every element $x$ of $S$ its inverse element
$x^{-1}$ is called the {\it inversion}.
On an inverse semigroup $S$ the semigroup operation  determines the following partial order $\preccurlyeq$: $s\preccurlyeq t$ if and only if there exists $e\in E(S)$ such that $s=te$. This partial order is called the \emph{natural partial order} on $S$.

\begin{definition}
Let $X$, $Y$ and $Z$ be topological spaces. A map $f\colon X\times Y\to Z$, $(x,y)\mapsto f(x,y)$, is called
\begin{itemize}
  \item[$(i)$] \emph{right} [\emph{left}] \emph{continuous} if it is continuous in the right [left] variable; i.e., for every fixed $x_0\in X$ [$y_0\in Y$] the map $Y\to Z$, $y\mapsto f(x_0,y)$ [$X\to Z$, $x\mapsto f(x,y_0)$] is continuous;
  \item[$(ii)$] \emph{separately continuous} if it is both left and right continuous;
  \item[$(iii)$] \emph{jointly continuous} if it is continuous as a map between the product space $X\times Y$ and the space $Z$.
\end{itemize}
\end{definition}

\begin{definition}[\cite{Carruth-Hildebrant-Koch=1983, Ruppert=1984}]
Let $S$ be a non-void topological space which is provided with an associative multiplication (a semigroup operation) $\mu\colon S\times S\to S$, $(x,y)\mapsto \mu(x,y)=xy$. Then the pair $(S,\mu)$ is called
\begin{itemize}
  \item[$(i)$] a \emph{right topological semigroup} if the map $\mu$ is right continuous, i.e., all interior left shifts $\lambda_s\colon S\to S$, $x\mapsto sx$, are continuous maps, $s\in S$;
  \item[$(ii)$] a \emph{left topological semigroup} if the map $\mu$ is left continuous, i.e., all interior right shifts $\rho_s\colon S\to S$, $x\mapsto xs$, are continuous maps, $s\in S$;
  \item[$(iii)$] a \emph{semitopological semigroup} if the map $\mu$ is separately continuous;
  \item[$(iv)$] a \emph{topological semigroup} if the map $\mu$ is jointly continuous.
\end{itemize}

We usually omit the reference to $\mu$ and write simply $S$ instead of $(S,\mu)$. It goes without saying that every topological semigroup is
also semitopological and every semitopological semigroup is both a right and left topological semigroup.
\end{definition}

A topology $\tau$ on a semigroup $S$ is called:
\begin{itemize}
  \item a \emph{semigroup} topology if $(S,\tau)$ is a topological semigroup;
  \item an \emph{inverse semigroup} topology if $(S,\tau)$ is an inverse topological semigroup with continuous inversion;
  \item a \emph{shift-continuous} topology if $(S,\tau)$ is a semitopological semigroup;
  \item a \emph{left-continuous} topology if $(S,\tau)$ is a left topological semigroup;
  \item a \emph{right-continuous} topology if $(S,\tau)$ is a right topological semigroup.
\end{itemize}

The bicyclic monoid ${\mathscr{C}}(p,q)$ is the semigroup with the identity $1$ generated by two elements $p$ and $q$ subjected only to the condition $pq=1$. The semigroup operation on ${\mathscr{C}}(p,q)$ is determined as
follows:
\begin{equation*}
    q^kp^l\cdot q^mp^n=
    \left\{
      \begin{array}{ll}
        q^{k-l+m}p^n, & \hbox{if~} l<m;\\
        q^kp^n,       & \hbox{if~} l=m;\\
        q^kp^{l-m+n}, & \hbox{if~} l>m.
      \end{array}
    \right.
\end{equation*}
It is well known that the bicyclic monoid ${\mathscr{C}}(p,q)$ is a bisimple (and hence simple) combinatorial $E$-unitary inverse semigroup and every non-trivial congruence on ${\mathscr{C}}(p,q)$ is a group congruence \cite{Clifford-Preston=1961}.

\smallskip

It is well known that topological algebra studies the influence of
topological properties of its objects on their algebraic properties
and the influence of algebraic properties of its objects on their
topological properties. There are two main problems in topological
algebra: the problem of non-discrete topologization and the problem
of embedding into objects with some topological-algebraic
properties.

\smallskip

In mathematical literature the question about non-discrete
(Hausdorff) topologization of groups was posed by Markov \cite{Markov=1945}.
Pontryagin gave well known conditions a base at the unity of a group
for its non-discrete topologization (see Theorem~3.9 of \cite{Pontryagin=1966}).
In \cite{Olshansky=1980}  Ol'shanskiy constructed an infinite
countable group $G$ such that every Hausdorff group topology on $G$
is discrete. Taimanov presented in \cite{Taimanov=1973} a commutative semigroup $\mathfrak{T}$ which admits only discrete Hausdorff semigroup topology and gave in \cite{Taimanov=1975} sufficient conditions on a commutative semigroup to have a non-discrete semigroup topology. In \cite{Gutik=2016} it is proved that each $T_1$-topology with continuous shifts on $\mathfrak{T}$ is discrete.
The bicyclic monoid admits only the discrete semigroup Hausdorff topology \cite{Eberhart-Selden=1969}. Bertman and  West in \cite{Bertman-West=1976} extended this result for the case of Hausdorff semitopological semigroups.

\smallskip

In the paper \cite{Chornenka-Gutik=2023} we construct two non-discrete inverse semigroup $T_1$-topologies and a compact inverse shift-continuous $T_1$-topology on the bicyclic monoid ${\mathscr{C}}(p,q)$. Also we give  conditions on a $T_1$-topology $\tau$ on ${\mathscr{C}}(p,q)$ to be discrete. In particular, we show that if $\tau$ is an inverse semigroup $T_1$-topology on ${\mathscr{C}}(p,q)$ which satisfies one of the following conditions: $\tau$ is Baire, $\tau$ is quasi-regular or $\tau$ is semiregular, then $\tau$ is discrete.

\smallskip

Subsemigroups of then bicyclic monoid were studied in \cite{Descalco-Ruskuc-2005, Descalco-Ruskuc-2008, Makanjuola-Umar=1997}.
In \cite{Makanjuola-Umar=1997}  the following anti-isomorphic subsemigroups of the bicyclic monoid
\begin{equation*}
  \mathscr{C}_{+}(a,b)=\left\{b^ia^j\in\mathscr{C}(a,b)\colon i\leqslant j,\, i,j\in\omega\right\}
  \end{equation*}
and
\begin{equation*}
  \mathscr{C}_{-}(a,b)=\left\{b^ia^j\in\mathscr{C}(a,b)\colon i\geqslant j,\, i,j\in\omega\right\}
\end{equation*}
are studied. In the paper \cite{Gutik=2023} topologizations of the semigroups $\mathscr{C}_{+}(a,b)$ and $\mathscr{C}_{-}(a,b)$ are studied. In particular in \cite{Gutik=2023} it proved that every Hausdorff left-continuous (right-continuous) topology on $\mathscr{C}_{+}(a,b)$ ($\mathscr{C}_{-}(a,b)$) is discrete and there exists a compact Hausdorff topological monoid $S$ which contains $\mathscr{C}_{+}(a,b)$ ($\mathscr{C}_{-}(a,b)$) as a submonoid. Also, a non-discrete right-continuous (left-continuous) topology $\tau_+^p$ ($\tau_-^p$) on the semigroup $\mathscr{C}_{+}(a,b)$ ($\mathscr{C}_{-}(a,b)$) which is not left-continuous (right-continuous).
In \cite{Gutik=2026} is proved that the monoid $\mathscr{C}_{+}(a,b)$ (resp., $\mathscr{C}_{-}(a,b)$)  contains a family $\{S_\alpha\colon \alpha\in\mathfrak{c}\}$ of continuum many subsemigroups with the following properties: $(i)$ every left-continuous (resp., right-continuous) Hausdorff topology on $S_\alpha$ is discrete; $(ii)$ every semigroup $S_\alpha$ admits a non-discrete right-continuous (resp., left-continuous) Hausdorff topology which is not left-continuous (resp., right-continuous).

\smallskip

In this paper we show that if a subsemigroup $S$ of the bicyclic monoid ${\mathscr{C}}(p,q)$  contains infinitely many idempotents then $S$ admits only the  discrete Hausdorff shift-continuous topology. Also we proof that every right-continuous (left-continuous\emph) Hausdorff Baire topology  on the semigroup $\mathscr{C}_+(a,b)$ $(\mathscr{C}_-(a,b))$ is discrete and the same statement holds for the bicyclic monoid.

\begin{theorem}\label{theorem-1}
Let $S$ be a subsemigroup of the bicyclic semigroup ${\mathscr{C}}(p,q)$. If $S$ contains infinitely many idempotents then every shift-continuous Hausdorff topology on $S$ is discrete.
\end{theorem}

\begin{proof}
Without loss of generality we may assume that the semigroup $S$ is infinite.

Fix an arbitrary element $b^ia^j$ of $S$.  Since the set $E(S)$ is infinite, there exists a positive integer $i_0$ such that $i_0\geqslant\max\{i,j\}+1$. Then the equalities
\begin{equation*}
  b^{i_0}a^{i_0}\cdot b^ka^l=
  \left\{
    \begin{array}{ll}
      b^ka^l,           & \hbox{if~} i_0\leqslant k;\\
      b^{i_0}a^{i_0-k+l}, & \hbox{if~} i_0>k
    \end{array}
  \right.
\end{equation*}
and
\begin{equation*}
  b^ka^l\cdot b^{i_0}a^{i_0}=
\left\{
    \begin{array}{ll}
      b^ka^l,           & \hbox{if~} i_0\leqslant l;\\
      b^{i_0}a^{i_0-l+k}, & \hbox{if~} i_0>l,
    \end{array}
  \right.
\end{equation*}
where $b^ka^l\in S$, imply that $A_{i_0}=S\setminus(Sb^{i_0}a^{i_0}\cup b^{i_0}a^{i_0}S)$ is infinite subset of $S$ and $b^{i}a^{j}\in A_{i_0}$. Also the above equalities imply that the mappings $\lambda_{i_0}\colon S\to S$, $b^ka^l\mapsto b^ka^l\cdot b^{i_0}a^{i_0}$ and $\rho_{i_0}\colon S\to S$, $b^ka^l\mapsto b^{i_0}a^{i_0}\cdot b^ka^l$ are retractions, and hence by \cite[1.5.C]{Engelking=1989} the set $A_{i_0}$ is open in $S$. This implies the point $b^ia^j$ has open finite neighbourhood in $S$, and hence it is an isolated point in the space. This completes the proof of the theorem.
\end{proof}

\begin{corollary}\label{corollary-2}
If $S$ is an inverse subsemigroup of the bicyclic semigroup ${\mathscr{C}}(p,q)$ then every shift-continuous Hausdorff topology on $S$ is discrete.
\end{corollary}

\begin{proof}
In the case when $S=E(S)$ the statement is trivial. Hence we assume that $S\neq E(S)$. Fix an arbitrary $b^ia^j\in S\setminus E(S)$. Without loss of generality we may assume that $i<j$. Since the semigroup $S$ is inverse, we obtain that $b^ja^i\in S$. Then for any positive integer $n$ the semigroup operation of the bicyclic semigroup implies that
\begin{align*}
  &(b^ia^j)^n=b^ia^{i+n(j-i)}\in S\setminus E(S), \\
  &(b^ja^i)^n=b^{i+n(j-i)}a^i\in S\setminus E(S),
\end{align*}
and hence
\begin{equation*}
  (b^ja^i)^n\cdot (b^ia^j)^n=b^{i+n(j-i)}a^i\cdot b^ia^{i+n(j-i)}=b^{i+n(j-i)}a^{i+n(j-i)}
\end{equation*}
is an idempotent of $S$ for any positive integer $n$. Next we apply Theorem~\ref{theorem-1}.
\end{proof}

We need the following proposition.

\begin{proposition}\label{proposition-3}
Let $S$ be an infinite subsemigroup of the bicyclic monoid $\mathscr{C}(a,b)$. If $S$ does not contain infinitely many idempotents, then either $S\subset \mathscr{C}_{+}(a,b)$ or $S\subset \mathscr{C}_{-}(a,b)$.
\end{proposition}

\begin{proof}
Suppose to the contrary that there exists an infinite subsemigroup $S$ of the bicyclic monoid $\mathscr{C}(a,b)$ such that $|E(S)|<\infty$, $(S\setminus E(S))\cap\mathscr{C}_{+}(a,b)\neq\varnothing$ and $(S\setminus E(S))\cap\mathscr{C}_{-}(a,b)\neq\varnothing$. Then there exist $b^ia^{i+k}\in S\cap \mathscr{C}_{+}(a,b)$ and $b^{j+l}a^j\in S\cap \mathscr{C}_{-}(a,b)$ for some $i,j,k,l\in\omega$ with $k,l>0$. Since $S$ is a subsemigroup of the bicyclic monoid $\mathscr{C}(a,b)$, the semigroup operation of $\mathscr{C}(a,b)$ implies that for any positive integer $p$ we have that
\begin{equation*}
  (b^ia^{i+k})^{lp}=b^ia^{i+klp}\in S \qquad \hbox{and} \qquad (b^{j+l}a^j)^{kp}=b^{j+klp}a^j\in S.
\end{equation*}
Hence we obtain that the following elements
\begin{equation*}
  b^ia^{i+klp}\cdot b^{j+klp}a^j=
  \left\{
    \begin{array}{ll}
      b^ja^j, & \hbox{if~} i<j; \\
      b^ia^j, & \hbox{if~} i=j; \\
      b^ia^i, & \hbox{if~} i>j
    \end{array}
  \right.
\end{equation*}
and
\begin{equation*}
   b^{j+klp}a^j\cdot b^ia^{i+klp}=
   \left\{
    \begin{array}{ll}
      b^{i+klp}a^{i+klp}, & \hbox{if~} j<i; \\
      b^{j+klp}a^{i+klp}, & \hbox{if~} j=i; \\
      b^{j+klp}a^{j+klp}, & \hbox{if~} j>i
    \end{array}
  \right.
\end{equation*}
are idempotents of $S$. Also by the last equality we get that the semigroup $S$ contains an infinite subset of idempotents $\left\{b^{i+klp}a^{i+klp}\colon p=1,2,3,\ldots\right\}$, a contradiction. The obtained contradiction implies the statement of the proposition.
\end{proof}

Next we define the $p$-adic topology on the set of integers $\mathbb{Z}$. Fix an arbitrary prime positive integer $p$. For any integer $a$ and any positive integer $k$ we put $U_k(a)=a+p^k\mathbb{Z}$. The topology $\tau_p$ which is generated by the base $\mathscr{B}_p=\left\{U_k(a)\colon a\in\mathbb{Z},k=1,2,3,\ldots\right\}$ is called the $p$-adic topology on $\mathbb{Z}$. It is well known that the additive group of integers with the $p$-idic topology $\tau_p$ is a non-discrete topological group \cite{Pontryagin=1966}. This implies that the additive semigroup of non-negative (resp. positive) integers $(\omega,+)$ (resp. $(\mathbb{N},+)$) with the induced topology from $(\mathbb{Z},\tau_p)$ is a non-discrete Hausdorff topological semigroup which we denote by $\tau_p$. It is obvious that the family $\mathscr{B}_p=\left\{V_k(a)\colon a\in\mathbb{Z},k=1,2,3,\ldots\right\}$, where $V_k(a)=a+p^k\omega$ is a base of the topology $\tau_p$ on $(\omega,+)$ ($(\mathbb{N},+)$).

We observe that there exist a non-discrete right-continuous (left-continuous) topology  $\tau_p^+$    ($\tau_p^-$) on the semigroup $\mathscr{C}_{+}(a,b)$    ($\mathscr{C}_{-}(a,b)$) which is not left-continuous (right-continuous) \cite{Gutik=2023}. The topology  $\tau_p^+$ on $\mathscr{C}_{+}(a,b)$ is constructed in the following way. The semigroup operation of $\mathscr{C}_{+}(a,b)$ implies that for any non-negative integer $n$ that
\begin{equation*}
  \mathscr{C}_{+}^n(a,b)=\left\{b^na^{n+i}\colon i\in\omega\right\}
\end{equation*}
is a subsemigroup of $\mathscr{C}_{+}(a,b)$. Moreover the semigroup $\mathscr{C}_{+}^n(a,b)$ is isomorphic to the additive semigroup of non-negative integers $(\omega,+)$ by the mapping $\mathfrak{I}_n\colon \omega,\to \mathscr{C}_{+}^n(a,b)$, $i\mapsto b^na^{n+i}$. Then for any $b^na^{n+i}\in \mathscr{C}_{+}(a,b)$ the mapping $\mathfrak{I}_n$ generates the base of the topology $\tau_p^+$  at the point $b^na^{n+i}$ as the image of the base $\mathscr{B}_p(i)$ of the topology $\tau_p$ at the point $i$ \cite{Gutik=2023}. The topology $\tau_p^-$  on the semigroup $\mathscr{C}_{-}(a,b)$ is constructed by the dual way.

Theorem~\ref{theorem-1} and Proposition~\ref{proposition-3} motivate to pose the following question.

\begin{question}
Let $S$ be a subsemigroup of the monoid $\mathscr{C}_{+}(a,b)$ which has no infinitely many idempotents. Does $S$ admit shift-continuous (semigroup) Hausdorff topology?
\end{question}

\begin{example}
Fix and arbitrary $n,m\in\omega$ such that  $m\leqslant n$. We define
\begin{equation*}
  \mathscr{C}_{+}^{[m,n]}(a,b)= \bigcup_{k=m}^n \mathscr{C}_{+}^k(a,b).
\end{equation*}
The semigroup operation of $\mathscr{C}_{+}(a,b)$ implies that $\mathscr{C}_{+}^{[m,n]}(a,b)$ is a subsemigroup of $\mathscr{C}_{+}(a,b)$. Also it is obvious that the semigroup $\mathscr{C}_{+}^{[m,n]}(a,b)$ is isomorphic to the monoid $\mathscr{C}_{+}^{[0,n-m]}(a,b)$ by the mapping $b^sa^{s+i}\mapsto b^{s-m}a^{s-m+i}$.

For an arbitrary prime positive integer $p$ we define a topology $\tau_p^{m,n}$ on $\mathscr{C}_{+}^{[m,n]}(a,b)$ in the following way. For any $b^ia^{i+j}\in\mathscr{C}_{+}^{[m,n]}(a,b)$ with $i+j\leqslant n$ the point $b^ia^{i+j}$ is isolated in $(\mathscr{C}_{+}^{[m,n]}(a,b),\tau_p^{m,n})$. If $i+j>n$ then the family
\begin{equation*}
  \mathscr{B}_p^{m,n}(b^ia^{i+j})=\left\{V_s(b^ia^{i+j})\colon s\in\mathbb{N}\right\},
\end{equation*}
where $V_s(b^ia^{i+j})=\left\{b^ia^{i+j+t}\colon t\in p^s\omega\right\}$, is a base of the topology $\tau_p^{m,n}$ at the point $b^ia^{i+j}$. It is obvious  that $\tau_p^{m,n}$ is a Hausdorff non-discrete topology on $\mathscr{C}_{+}^{[m,n]}(a,b)$.
\end{example}

\begin{proposition}\label{proposition-4}
$\tau_p^{m,n}$ is a semigroup topology on $\mathscr{C}_{+}^{[m,n]}(a,b)$.
\end{proposition}

\begin{proof}
Fix arbitrary $b^{i_1}a^{i_1+j_1},b^{i_2}a^{i_2+j_2}\in\mathscr{C}_{+}^{[m,n]}(a,b)$. Then we have that
\begin{equation*}
  b^{i_1}a^{i_1+j_1}\cdot b^{i_2}a^{i_2+j_2}=
  \left\{
    \begin{array}{ll}
      b^{i_2-j_1}a^{i_2+j_2}, & \hbox{if~} i_1+j_1<i_2;\\
      b^{i_1}a^{i_2+j_2},     & \hbox{if~} i_1+j_1=i_2;\\
      b^{i_1}a^{i_1+j_1+j_2}, & \hbox{if~} i_1+j_1>i_2.
    \end{array}
  \right.
\end{equation*}

We consider all possible cases.

Suppose that $i_1+j_1\leqslant n$ and $i_2+j_2\leqslant n$. Then $b^{i_1}a^{i_1+j_1}$ and $b^{i_2}a^{i_2+j_2}$ are isolated points in the topological space $(\mathscr{C}_{+}^{[m,n]}(a,b),\tau_p^{m,n})$, and hence in this case the semigroup oparation is continuous.

Suppose that $i_1+j_1\leqslant n$ and $i_2+j_2>n$. Then $b^{i_1}a^{i_1+j_1}$ is isolated point in $(\mathscr{C}_{+}^{[m,n]}(a,b),\tau_p^{m,n})$. Simple verifications show that for any positive integer $s$ we have that
\begin{align*}
  b^{i_1}a^{i_1+j_1}\cdot V_s(b^{i_2}a^{i_2+j_2})&= b^{i_1}a^{i_1+j_1}\cdot \left\{b^{i_2}a^{i_2+j_2+t}\colon t\in p^s\omega\right\}=\\
  &=
\left\{
    \begin{array}{ll}
      \left\{b^{i_2-j_1}a^{i_2+j_2+t}\colon t\in p^s\omega\right\}, & \hbox{if~} i_1+j_1<i_2;\\
      \left\{b^{i_1}a^{i_2+j_2+t}\colon t\in p^s\omega\right\},     & \hbox{if~} i_1+j_1=i_2;\\
      \left\{b^{i_1}a^{i_1+j_1+j_2+t}\colon t\in p^s\omega\right\}, & \hbox{if~} i_1+j_1>i_2
    \end{array}
  \right.=\\
  &=
\left\{
    \begin{array}{ll}
      V_s(b^{i_2-j_1}a^{i_2+j_2}), & \hbox{if~} i_1+j_1<i_2;\\
      V_s(b^{i_1}a^{i_2+j_2}),     & \hbox{if~} i_1+j_1=i_2;\\
      V_s(b^{i_1}a^{i_1+j_1+j_2}), & \hbox{if~} i_1+j_1>i_2.
    \end{array}
  \right.
\end{align*}

Suppose that $i_1+j_1> n$ and $i_2+j_2\leqslant n$. Then $b^{i_2}a^{i_2+j_2}$ is isolated point in $(\mathscr{C}_{+}^{[m,n]}(a,b),\tau_p^{m,n})$ and $i_2<i_1+j_1$. By usual calculations for any positive integer $s$ we get that
\begin{align*}
  V_s(b^{i_1}a^{i_1+j_1})\cdot b^{i_2}a^{i_2+j_2}&=\left\{b^{i_1}a^{i_1+j_1+t}\colon t\in p^s\omega\right\}\cdot b^{i_2}a^{i_2+j_2}= \\
   &=\left\{b^{i_1}a^{i_1+j_1+j_2+t}\colon t\in p^s\omega\right\}=\\
   &=V_s(b^{i_1}a^{i_1+j_1+j_2}).
\end{align*}

Suppose that $i_1+j_1> n$ and $i_2+j_2>n$. Then $i_2<i_1+j_1$.  By usual calculations for any positive integer $s$ we have that
\begin{align*}
  V_s(b^{i_1}a^{i_1+j_1})\cdot V_s(b^{i_2}a^{i_2+j_2})&
  =\left\{b^{i_1}a^{i_1+j_1+t_1}\colon t_1\in p^s\omega\right\}\cdot \left\{b^{i_2}a^{i_2+j_2+t_2}\colon t_2\in p^s\omega\right\}= \\
   &=\left\{b^{i_1}a^{i_1+j_1+t_1+j_2+t_2}\colon t_1,t_2\in p^s\omega\right\}\subseteq \\
   &\subseteq \left\{b^{i_1}a^{i_1+j_1+j_2+t}\colon t\in p^s\omega\right\}=\\
   &=V_s(b^{i_1}a^{i_1+j_1+j_2}).
\end{align*}

The above arguments imply the statement of the proposition.
\end{proof}

We recall that a topological space $X$ is said to be  \emph{Baire} if for each sequence $A_1, A_2,\ldots, A_i,\ldots$ of dense open subsets of $X$ the intersection $\displaystyle\bigcap_{i=1}^\infty A_i$ is a dense subset of $X$ \cite{Haworth-McCoy=1977}.

\begin{theorem}\label{theorem-5}
Every right-continuous \emph{(}left-continuous\emph{)} Hausdorff Baire topology $\tau$ on the semigroup $\mathscr{C}_+(a,b)$ $(\mathscr{C}_-(a,b))$ is discrete.
\end{theorem}

\begin{proof}
We shall prove the statement of the theorem only for the semigroup $\mathscr{C}_+(a,b)$, because the semigroups $\mathscr{C}_+(a,b)$ and $\mathscr{C}_-(a,b)$ are anti-isomorphic \cite{Gutik=2023, Makanjuola-Umar=1997}.

Fix an arbitrary $b^{i_0}a^{j_0}\in\mathscr{C}_+(a,b)$. Since every left shift on $(\mathscr{C}_+(a,b),\tau)$ is continuous and $b^{i_0+1}a^{i_0+1}$ is an idempotent of $\mathscr{C}_+(a,b)$, the mapping $\lambda_{b^{j_0+1}a^{j_0+1}}\colon \mathscr{C}_+(a,b)\to\mathscr{C}_+(a,b)$, $b^sa^t\mapsto b^{j_0+1}a^{j_0+1} \cdot b^sa^t$ is a continuous retraction. Then by \cite[1.5.C]{Engelking=1989} the retract $b^{j_0+1}a^{j_0+1}\mathscr{C}_+(a,b)$ is a closed subset of the topological space $(\mathscr{C}_+(a,b),\tau)$. It is obvious that $b^{i_0}a^{j_0}\notin b^{j_0+1}a^{j_0+1}\mathscr{C}_+(a,b)$.

We define
\begin{equation*}
A_{j_0+1}=\left\{b^ia^j\in\mathscr{C}_+(a,b)\colon i+j\leqslant 2(j_0+1) \right\}.
\end{equation*}
Then $A_{j_0+1}$ is a finite subset of $\mathscr{C}_+(a,b)$ and $b^{i_0}a^{j_0}\in A_{j_0+1}$. Since the space $(\mathscr{C}_+(a,b),\tau)$ is Hausdorff,
\begin{equation*}
S=\mathscr{C}_+(a,b)\setminus (A_{j_0+1}\cup b^{j_0+1}a^{j_0+1}\mathscr{C}_+(a,b))
\end{equation*}
is open in $(\mathscr{C}_+(a,b),\tau)$, and hence by Proposition~1.14 of \cite{Haworth-McCoy=1977} the space $S$ is Baire. By Proposition~1.30 of \cite{Haworth-McCoy=1977} the space $S$ contains infinitely many isolated points in $S$, because the set $S$ is infinite and countable. Then there exists a non-negative integer $x_0\leqslant j_0$ such that the set $S_{x_0}=\{b^{x_0}a^y\colon y\geqslant x_0\}$ contains infinitely many isolated points of $S$. This implies that there exists a positive integer $y_0$ such that $y_0-j_0>x_0\geqslant 0$, and hence $b^xa^{y_0+i_0-j_0}\in \mathscr{C}_+(a,b)$. The semigroup operation of $\mathscr{C}_+(a,b)$ implies that
\begin{equation*}
  b^{x_0}a^{y_0+i_0-j_0} \cdot b^{i_0}a^{j_0}=b^{x_0}a^{y_0},
\end{equation*}
because $y_0+i_0-j_0>x_0+i_0>i_0$. Since $(\mathscr{C}_+(a,b),\tau)$ is a left topological semigroup, we have that the set of solutions $U$ of the equation
\begin{equation*}
  b^{x_0}a^{y_0+i_0-j_0} \cdot X=b^{x_0}a^{y_0}
\end{equation*}
is an open subset of $(\mathscr{C}_+(a,b),\tau)$ which contains the point $b^{i_0}a^{j_0}$. By Lemma I.1.$(ii)$ of \cite{Eberhart-Selden=1969} the set $U$ is finite. Since $(\mathscr{C}_+(a,b),\tau)$ is a Hausdorff space, the point $b^{i_0}a^{j_0}$ is isolated in $(\mathscr{C}_+(a,b),\tau)$. This completes the proof of the theorem.
\end{proof}

A topological space $X$ is called \emph{locally compact}, if for any point $x\in X$ there exists an open neighbourhood $U(x)$ such that the closure $\operatorname{cl}_X(U(x))$ of $U(x)$ is a compact set \cite{Engelking=1989}.
Since every locally compact space is Baire \cite{Engelking=1989}, Theorem~\ref{theorem-5} implies the following corollary.

\begin{corollary}\label{corollary-6}
Every right-continuous \emph{(}left-continuous\emph{)} Hausdorff locally compact topology on the semigroup $\mathscr{C}_+(a,b)$ $(\mathscr{C}_-(a,b))$ is discrete.
\end{corollary}

\begin{remark}
In \cite{Gutik=2023} a non-discrete non-Baire Hausdorff topology $\tau^+_p$ on the semigroup $\mathscr{C}_+(a,b)$ such that $(\mathscr{C}_+(a,b),\tau^+_p)$ is a metrizable right topological semigroup is constructed.
\end{remark}

Theorem~\ref{theorem-7} extends results of Theorem~1 from \cite{Chornenka-Gutik=2023} onto Hausdorff right topological and left topological semigroups.

\begin{theorem}\label{theorem-7}
Every right-continuous \emph{(}left-continuous\emph{)} Hausdorff Baire topology $\tau$ on the bicyclic semigroup $\mathscr{C}(a,b)$ is discrete.
\end{theorem}

The proof of Theorem~\ref{theorem-7} is similar to Theorem~\ref{theorem-5}.

\smallskip

Theorem~\ref{theorem-7} implies

\begin{corollary}\label{corollary-8}
Every right-continuous \emph{(}left-continuous\emph{)} Hausdorff locally compact topology on the bicyclic semigroup $\mathscr{C}(a,b)$ is discrete.
\end{corollary}


\end{document}